\documentclass{proc-l-hijacked}
\usepackage{amssymb, amsmath, amsthm, hyperref,}

\newtheorem{theorem}{Theorem}[section]
\newtheorem{lemma}[theorem]{Lemma}

\theoremstyle{definition}

\theoremstyle{remark}
\newtheorem{remark}[theorem]{Remark}

\numberwithin{equation}{section}

\DeclareMathOperator{\Ima}{Im}
\DeclareMathOperator{\Rea}{Re}
\begin{document}
\title[]{Some character generating functions on Banach algebras.}
\author{ C. Tour\'e, F. Schulz and R. Brits}
\address{Department of Mathematics, University of Johannesburg, South Africa}
\email{cheickkader89@hotmail.com, francoiss@uj.ac.za, rbrits@uj.ac.za}
\subjclass[2010]{15A60, 46H05, 46H10, 46H15, 47B10}
\keywords{Banach algebra, spectrum, character, linear functional}

\begin{abstract}
We consider a multiplicative variation on the classical Kowalski-S\l{}odkowski Theorem which identifies the characters among the collection of all functionals on a Banach algebra $A$. In particular we show that, if $A$ is a $C^*$-algebra, and if $\phi:A\mapsto\mathbb C$ is a continuous function satisfying $\phi(\mathbf 1)=1$ and $\phi(x)\phi(y) \in \sigma(xy)$ for all $x,y\in A$ (where $\sigma$ denotes the spectrum), then $\phi$ generates a corresponding character $\psi_\phi$ on $A$ which coincides with $\phi$ on the principal component of the invertible group of $A$. We also show that, if $A$ is any Banach algebra whose elements have totally disconnected spectra, then, under the aforementioned conditions, $\phi$ is always a character.  
\end{abstract}
\parindent 0mm

\maketitle

\section{Introduction}\label{intro}

In this paper $A$ will always be a complex and unital Banach algebra, with the unit denoted by $\mathbf 1$. The invertible group of $A$ will be denoted by $G(A)$, and the connected component of $G(A)$ containing $\mathbf 1$, by $G_{\mathbf 1}(A)$. It is  well known (see for instance \cite[Theorem 3.3.7]{aupetit1991primer}) that
\begin{equation}\label{G1} 
G_{\mathbf 1}(A)=\{e^{x_1}\cdots e^{x_k}: k\in\mathbb N,\ x_j\in A\}.
\end{equation}
 If $x\in A$ then the spectrum of $x$ is the (necessarily non-empty and compact) set $\sigma(x):=\{\lambda\in\mathbb C:\lambda\mathbf 1-x\notin G(A)\}.$
 A character of $A$ is, by definition, a linear functional $\chi:A\rightarrow\mathbb C$ which is simultaneously  multiplicative i.e. $\chi(xy)=\chi(x)\chi(y)$ holds for all $x,y\in A$.  
 Depending on the specific algebra, or class of algebras, characters may or may not exist. One immediately recalls the famous result of Gleason, Kahane, and \.{Z}elazko, \cite{Gleason,ZelazkoKahane,Zelazko} which identifies the characters among the dual space members of $A$ via a spectral condition:
\begin{theorem}[Gleason-Kahane-\.{Z}elazko]\label{origin}
	Let $A$ be a Banach algebra. Then $\phi\in A^\prime$, the dual of $A$, is a character of $A$ if and only if $\phi(x)\in\sigma(x)$ for each $x\in A$. 
\end{theorem}
A perhaps lesser known, but stronger result, due to Kowalski and S\l{}odkowski \cite{Slodkowski}, identifies the characters among all complex-valued functions on $A$ via a spectral condition:

\begin{theorem}[Kowalski-S\l{}odkowski]\label{allfunctions}
Let $A$ be a Banach algebra. Then a function $\phi:A\rightarrow\mathbb C$ is a character of $A$ if and only if $\phi$  satisfies
\begin{itemize}
	\item[(i)]{$\phi(0)=0$,}
	\item[(ii)]{$\phi(x)-\phi(y)\in\sigma(x-y)$ for every $x,y\in A$. }
\end{itemize} 
	
\end{theorem}

\begin{remark}
 It is easy to see that the Kowalski-S\l{}odkowski Theorem can be more economically formulated as $$\phi(x)+\phi(y)\in\sigma(x+y) \mbox{ for every } x, y\in A\Leftrightarrow\phi\mbox{ is a character of }A.$$
\end{remark}
	
One is now naturally led to ask whether there exist ``multiplicative" versions of Theorem~\ref{origin} and Theorem~\ref{allfunctions}. That is, under what conditions is a function with multiplicative properties, perhaps involving the spectrum, a character? The problem seems thorny, but some positive results were obtained in 
\cite{formesmultiplicatives} and \cite{TOURE2}.

\begin{theorem}[{\cite[Corollary 2.2]{TOURE2}}]\label{toure}
	Let $A$ be a Banach algebra. Then a multiplicative function $\phi:A\rightarrow\mathbb C$ satisfying $\phi(x)\in\sigma(x)$ for each $x\in A$ is a character if and only if for each $x\in A$ the map 
	\begin{equation}\label{subharmonic}
	\lambda\mapsto\left|\phi(x-\lambda\mathbf 1)+\lambda\right|
	\end{equation}
	is subharmonic on $\mathbb C$.
\end{theorem}
One may show (\cite[Theorem 2.1]{TOURE2}) that the subharmonic condition on\eqref{subharmonic} in Theorem \ref{toure} can be replaced with the requirement that $\lambda\mapsto\phi(x-\lambda\mathbf 1)$ is an entire function for each $x\in A$.
\begin{theorem}[Maouche]
Let $A$ be a Banach algebra, and let $\phi:A\rightarrow\mathbb C$ be a multiplicative function satisfying $\phi(x)\in\sigma(x)$ for each $x\in A$. Then, corresponding to $\phi$, there exists a unique character on $A$ which agrees with $\phi$ on $G_{\mathbf 1}(A)$. 
\end{theorem}

Among a number of results for $C^\star$-algebras (on the above-mentioned topic), it is further shown, in \cite{TOURE2}, that for the particular case of von Neumann algebras, a continuous multiplicative function with values  $\phi(x)\in\sigma(x)$, for each $x\in A$, is always a character.  

The current paper is motivated by \emph{multiplicatively} spectrum preserver problems which were studied already in 1997 by B. Aupetit in \cite[Theorem 3.5, p.74]{banach97}, and later also in \cite{unitalspec, molnar1, raoroy, raoroy2}, as well as the Kowalski-S\l{}odkowski Theorem. We shall consider a function $\phi:A\rightarrow\mathbb C$ (not assumed to be linear or multiplicative) satisfying the following conditions:
\begin{itemize}
	\item[(P1)] $\phi(x)\phi(y) \in \sigma(xy)$ for all $x,y\in A$,
	\item[(P2)] $\phi(\mathbf 1)=1$,
	\item[(P3)] $\phi$ is  continuous on $A$.
\end{itemize}

In \cite{unitalspec} Hatori \emph{et. al.} show that a multiplicative Kowalski-S\l{}odkowski Theorem is generally not possible. In particular, even for commutative $C^*$-algebras, the conditions (P1)--(P2) are not enough to guarantee that $\phi$ is a character (see also \cite[p.56]{TOURE2} and \cite[p.44--45]{formesmultiplicatives}). What we want to show here is that some positive results can be obtained (with (P3) added to the list) for at least two classes of Banach algebras, one of which is general $C^*$-algebras. Our proofs rely on the Lie-Trotter Formula, stated below, and the additive Kowalski-S\l{}odkowski Theorem stated above. The current paper seems to be the first attempt to address the issue raised by Hatori \emph{et. al.}, and so the paper is essentially self contained.

Obviously, if $\phi$ satisfies (P1)--(P2), then
\begin{equation}\label{inspec}
x\in A\Rightarrow \phi\left(x\right)\in\sigma(x)
\end{equation}
and
\begin{equation}\label{obs}
x\in G(A)\Rightarrow \phi\left(x^{-1}\right)=\phi(x)^{-1}
\end{equation}
from which it follows that
\begin{equation}\label{obs2}
x\in G(A)\Rightarrow \phi\left(\lambda x\right)=\lambda\phi(x)\mbox{ for all }\lambda\in\mathbb C.
\end{equation}
 
In the remainder of this paper we shall make use of the classical: 

\begin{theorem}[Lie-Trotter Formula, {\cite[p.67]{aupetit1991primer}}]\label{LieTrot}
Let $A$ be a Banach algebra, and let $x,y\in A$. Then
$$\lim_{n\rightarrow\infty}\left(e^{x/n}e^{y/n}\right)^n=e^{x+y}.$$	
\end{theorem}

\section{Totally disconnected spectra}\label{disconnect}

Throughout this section, $A$ is a complex and unital Banach algebra for which $\sigma (x)$ is totally disconnected for each $x \in A$, and $\phi:A\rightarrow\mathbb C$ is a map which satisfies the properties (P1)--(P3) in Section~\ref{intro}. 

\begin{lemma}\label{r1}
	Let $x \in A$. Then:
	\begin{itemize}
		\item[\textnormal{(i)}]
		$\phi\left(\lambda \mathbf{1} + x\right) = \lambda + \phi(x)$ for each $\lambda \in \mathbb{C}$.
		\item[\textnormal{(ii)}]
		$\phi\left(\lambda x\right) = \lambda\phi(x)$ for each $\lambda \in \mathbb{C}$.
	\end{itemize}
\end{lemma}

\begin{proof}
(i) If we define $\alpha_\lambda:=\phi(\lambda \mathbf 1 +x)-\lambda,$ then it follows that  $\alpha_\lambda$ is a continuous function on $\mathbb C$ with values belonging to  $\sigma(x)$. So, since $ \mathbb C$ is connected and $\sigma(x)$ is totally disconnected, we infer that $ \phi(\lambda \mathbf1+x)-\lambda$ is constant on $\mathbb C$. Thus $\phi(\lambda \mathbf1+x)=\lambda+\phi(x)$ holds for all $\lambda\in\mathbb C$.\\
(ii) If we define $\alpha_\lambda :=\phi(\lambda x)/\lambda,$ then it follows that  $\alpha_\lambda$ is a continuous function on  $\mathbb C\setminus\{0\}$ with values belonging to  $\sigma(x)$. Using a similar reasoning as in (i) we infer that $\phi(\lambda x)=\lambda\phi( x)$.
\end{proof}

\begin{lemma}\label{r2}
	Let $r \in A$ and suppose that $\sigma(r) = \left\{1, k \right\}$, where $k\not=0$ and $k \neq 1$. If $\phi(r) = 1$, then $\phi \left(r^{n}\right) = 1$ for all $n \in \mathbb{N}$.
\end{lemma}

\begin{proof}
	Suppose, for the sake of a contradiction, that $\phi\left(r^{n}\right) \neq 1$ for some $n \in \mathbb{N}$. By the Spectral Mapping Theorem, $\sigma\left(r^{m}\right) = \left\{1, k^{m} \right\}$ for all $m \in \mathbb{N}$. Hence, $\phi\left(r^{n}\right) = k^{n}$. However, since
	$$k^{n} = \phi\left(r^{n}\right) = \phi\left(r^{n}\right)\phi(r) \in \sigma\left(r^{n+1}\right),$$
	we obtain $k^{n} = k^{n+1}$ or $k^{n} = 1$ which gives a contradiction. 
\end{proof}

\begin{lemma}\label{r4}
	$\phi$ satisfies:
		\begin{itemize}
			\item[\textnormal{(i)}]
 For each $x \in G(A)$, $\phi\left(x^{n}\right) = \phi(x)^{n}$ for all $n \in \mathbb{N}$.
			\item[\textnormal{(ii)}]
 For each $x \in A$, $\phi \left(e^{x}\right) = e^{\phi(x)}$.
		\end{itemize}

\end{lemma}

\begin{proof}
(i)	 Let $x \in G(A)$ and suppose that $\phi(x) = \alpha$. For the sake of a contradiction, assume that $\phi\left(x^{n}\right) \neq \alpha^{n}$ for some $n \in \mathbb{N}$. By the Spectral Mapping Theorem, and the fact that $\phi\left(x^{n}\right) \in \sigma \left(x^{n}\right)$, it follows that $\phi\left(x^{n}\right) = \beta^{n}$ for some $\beta \in \sigma(x)$ with $\alpha \neq \beta$. Let $\left|\beta^{n} - \alpha^{n}\right| =: \epsilon > 0$. By continuity of the polynomial $g(\lambda) = \lambda^{n}$, there exists a $\delta > 0$ such that $\left|\lambda - \alpha\right| <\delta$ implies $\left|\lambda^{n} - \alpha^{n}\right| < \epsilon$. Since $\sigma (x)$ is totally disconnected it is well-known that we can find disjoint open sets $U_{1}, \ldots, U_{m}$ in $\mathbb{C}$ such that each $U_{j}$ has diameter less than $\delta$ and $\sigma(x) \subseteq U:= U_{1} \cup \cdots \cup U_{l}$. Without loss of generality we may assume that $\alpha \in U_{1}$. Since $U_{1}$ has diameter less than $\delta$ it follows that $\beta \in U\setminus U_{1}$. Choose $k \in \mathbb{N}$ such that $k\alpha \notin \sigma(x)$ and note that $k \neq 1$. Now, let $f$ be the holomorphic function on $U$ defined by
	$$f(\lambda) = \left\{ \begin{array}{cl} 1 & \mbox{if }\lambda \in U_{1} \\ k & \mbox{if }\lambda \in U\setminus U_{1} \end{array} \right.$$ 
	By the Holomorphic Functional Calculus, it follows that $r:= f(x)$ has spectrum $\sigma(r) = \left\{1, k \right\}$. We claim that $\phi(r) = 1$: If $\phi(r) \neq 1$, then $\phi(r) = k$. However, then, by (P1) and the Holomorphic Functional Calculus, we obtain
	$$k\alpha = \phi(r)\phi(x) \in \sigma(rx) = \left\{\lambda: \lambda \in \sigma(x)\cap U_{1} \right\}  \cup \left\{k\lambda: \lambda \in \left(\sigma(x)\cap U\right)\setminus U_{1} \right\}.$$  
	But this yields a contradiction since $k\alpha \notin \sigma(x)$ and the $U_{j}$ are all disjoint. We may therefore conclude that $\phi(r) = 1$ as claimed. By Lemma \ref{r2} it therefore follows that $\phi\left(r^{m}\right) = 1$ for all $m \in \mathbb{N}$. Now, by the Holomorphic Functional Calculus, we have for each $m \in \mathbb{N}$ that
	\begin{equation}
	\sigma\left(r^{m}x^{n}\right) = \left\{\lambda^{n}: \lambda \in \sigma(x)\cap U_{1} \right\}  \cup \left\{k^{m}\lambda^{n}: \lambda \in \left(\sigma(x)\cap U\right)\setminus U_{1} \right\}.
	\label{eq1}
	\end{equation}
	Moreover, by the (P1) property of $\phi$, we get
	$$\beta^{n} = \phi\left(x^{n}\right) = \phi\left(r^{m}\right)\phi\left(x^{n}\right) \in \sigma\left(r^{m}x^{n}\right)$$
	for each $m \in \mathbb{N}$. Since $\lambda \in \sigma(x)\cap U_{1}$ implies that $\left|\lambda^{n} - \alpha^{n}\right| < \epsilon$, it must be the case that for each $m \in \mathbb{N}$, $\beta^{n}$ belongs to the second set in the union in (\ref{eq1}). Consequently, for each $m \in \mathbb{N}$ there exists a $\lambda_{m} \in \sigma(x)$ such that $\beta^{n} = k^{m}\lambda_{m}^{n}$. But then
	$$\lim_{m \rightarrow \infty} \left|\lambda_{m}\right|^{n} = \lim_{m \rightarrow \infty} {k^{m}\left|\lambda_{m}\right|^{n}}/{k^{m}} = \lim_{m \rightarrow \infty} {\left|\beta\right|^{n}}/{k^{m}} = 0 \mbox{ since }k>1,$$
	and so, $0 \in \sigma(x^{n})$. But this is absurd since $x \in G(A)$. We therefore conclude that, for each $x \in G(A)$, $\phi\left(x^{n}\right) = \phi(x)^{n}$ for all $n \in \mathbb{N}$. \\
(ii) 	Let $x \in A$. Then $\mathbf{1} + x/n \in G(A)$ for all sufficiently large $n \in \mathbb{N}$, say $n \geq m$. In particular, by hypothesis and Lemma~\ref{r1}, this means that if $n\geq m$
$$\phi \left( \left(\mathbf{1} + {x}/{n}\right)^{n}\right) = \phi\left(\mathbf{1} + {x}/{n}\right)^{n} = \left(1 + {\phi(x)}/{n}\right)^{n}.$$
By (P3) we therefore obtain that
$$\phi \left( \lim_{n \rightarrow \infty} \left(\mathbf{1} + {x}/{n}\right)^{n}\right) = \lim_{n \rightarrow \infty} \left(1 + {\phi(x)}/{n}\right)^{n}.$$
Hence, $\phi \left(e^{x}\right) = e^{\phi(x)}$ as desired.	
\end{proof}

\begin{theorem}\label{r5}
 $\phi$ is a character of $A$. 
\end{theorem}

\begin{proof}
	Let $x,y\in A$ be arbitrary, and let $n$ be a natural number. Using (P3), the Spectral Mapping Theorem, and Lemma~\ref{r4} we have 
	\begin{align*}
	\phi\left( e^{{x}/{n} }\right)\phi\left(e^{{y}/{n}}\right)\in\sigma\left(e^{{x}/{n}}e^{{y}/{n}}\right)&
	\Rightarrow\phi\left(e^{{x}/{n} }\right)^n\phi\left(e^{{y}/{n}}\right)^n \in \sigma\left(\left[e^{{x}/{n}}e^{{y}/{n}}\right]^n\right)\\&\Rightarrow\phi\left(e^x\right)\phi\left(e^y\right) \in \sigma\left(\left[e^{{x}/{n}}e^{{y}/{n}}\right]^n\right)\\&\Rightarrow
	e^{\phi(x)+\phi(y)}\in \sigma\left(\left[e^{{x}/{n}}e^{{y}/{n}}\right]^n\right)
	\end{align*}	
 Since $\sigma(a)$ is totally disconnected for each $a\in A$ the map $a\mapsto\sigma(a)$ is continuous at each $a\in A$, and so, by the Lie-Trotter Formula, we have
 \begin{equation}\label{chartd}
 e^{\phi(x)+\phi(y)}\in \sigma\left(e^{x+y}\right).
 \end{equation}
	 Let $m$ be a sufficiently large natural number such that the set ${\sigma((x+y)/m)}$ and the complex number ${(\phi(x)+\phi(y))}/{m}$ are both in the fundamental strip $\{\lambda\in\mathbb C:-\pi<\Ima(\lambda)\leq\pi\}$. Then, since the exponential function is injective on the strip, it follows from \eqref{chartd} with $x,y$ replaced by $x/m,y/m$ that  ${(\phi(x)+\phi(y))}/{m}\in{\sigma(x+y)}/{m}.$ This   implies that $\phi(x)+\phi(y) \in \sigma(x+y),$ and so, by the Kowalski-S\l{}odkowski Theorem, $\phi$ is a character. 
	
\end{proof}

In particular if $A$ is finite dimensional and $\phi:A\rightarrow\mathbb C$ satisfies (P1)--(P3), then $\phi$ is a character. 

\section{$C^\star$-algebras}
Throughout this section $A$ will denote a complex and unital $C^\star$-algebra, and $\mathcal S$ will denote the (real) Banach space of self-adjoint elements of $A$.
If $x\in A$ then we denote $$\Rea(x):=(x+x^\star)/2\mbox{ and }\Ima(x):=(x-x^\star)/2i.$$ As before $\phi:A\rightarrow\mathbb C$ is a map which satisfies the properties (P1)--(P3) in Section~\ref{intro}.
 We shall, from now on, also use \eqref{inspec} as well as the Spectral Mapping Theorem without further specific reference. Passing through a sequence of lemmas we derive the main result, Theorem~\ref{Cstar}

\begin{lemma}\label{lem31}
 Let $x \in\mathcal{S}$. If  $\phi(x) \neq 0$, then:
	\begin{itemize}
		\item[(i)] $\phi(\mathbf 1+ix)=1+i\phi(x),$
		\item[(ii)]$\phi(tx)=t\phi(x),$ for each $t\in\mathbb R,$
		\item[(iii)] $\phi(e^{tx})=e^{\phi(tx)}=e^{t\phi(x)},$ for each $t\in\mathbb R,$
		\item[(iv)] $\phi(x^n)=\phi(x)^n,$ for each $n\in\mathbb N$.
	\end{itemize}
\end{lemma}
\begin{proof} (i) From (P1) we get that $\phi(x)\phi(\mathbf 1+ix)\in\sigma(x+ix^2)$. Writing $\phi(\mathbf 1+ix)=1+i\alpha$ where $\alpha\in\sigma(x)$ it follows that $\phi(x)(1+i\alpha)=\beta+i\beta^2$ for some $\beta\in\sigma(x)$. Since $\sigma(x)\subset\mathbb R$ we infer that $\phi(x)=\alpha=\beta$, and hence that $\phi(\mathbf 1+ix)=1+i\phi(x)$.\\
(ii) Fix $t\in\mathbb R$, and let $\alpha>0$ (which we will regard as a variable). By (P1) 
$$\phi(x)\phi(\alpha i\mathbf 1+tx)\in\sigma\left(\alpha ix+tx^2\right),$$ from which we can write
\begin{equation}\label{f1}
\phi(x)\phi(\alpha i\mathbf 1+tx)=\beta_\alpha\alpha i+t\beta^2_\alpha\mbox{ with }\beta_\alpha\in\sigma(x).
\end{equation}
But we can also write $\phi(\alpha i\mathbf 1+tx)=\alpha i+t\lambda_\alpha$ so that
\begin{equation}\label{f2}
\phi(x)\phi(\alpha i\mathbf 1+tx)=\phi(x)\alpha i+t\phi(x)\lambda_\alpha\mbox{ with }\lambda_\alpha\in\sigma(x).
\end{equation}
Comparing the real and imaginary parts of \eqref{f1} and \eqref{f2} on the right, using the fact that $\phi(x)\not=0$, $\alpha\not=0$, and $\alpha,t,\phi(x),\beta_\alpha,\lambda_\alpha\in\mathbb R$ it follows that
$\phi(x)=\beta_\alpha=\lambda_\alpha$ for all $\alpha>0$. So, for all $\alpha>0$, we have 
$$\phi(\alpha i\mathbf 1+tx)=\alpha i+t\phi(x).$$
If we let $\alpha\rightarrow0$ then (P3) implies $\phi(tx)=t\phi(x).$
	
	(iii) By (ii) it suffices to show that $\phi(e^x)=e^{\phi(x)}.$  Consider $$\phi(e^x)\phi(\mathbf 1+ix)\in \sigma(e^x+ixe^x).$$ Then, using (i), $\phi(e^x)+i\phi(e^x)\phi(x)=e^\gamma+e^\gamma\gamma i$ for some $\gamma \in \sigma(x)$. Consequently $\phi(e^x)=e^\gamma$ and $\phi(x)=\gamma$ which implies $\phi(e^x)=e^{\phi(x)}.$\\
	(iv) Let $\alpha>0$ be a variable. Then, from (P1), $\phi(x)\phi\left(x^n+\alpha i \mathbf 1\right)=\beta_\alpha^{n+1}+\alpha i \beta_\alpha $, where $\beta_\alpha\in\sigma(x)$. On the other hand, we can also write $\phi(x^n+\alpha i\mathbf 1)=\lambda_\alpha^n+\alpha i$ where $\lambda_\alpha\in\sigma(x)$. Thus $$\phi(x)\lambda_\alpha^n+\alpha i\phi(x)=\beta_\alpha^{n+1}+\alpha i \beta_\alpha ,$$
	which implies that $\phi(x)=\beta_\alpha$ and $\lambda_\alpha^n=\phi(x)^n $. Therefore 
	$$\phi(x^n+\alpha i)=\phi(x)^n+\alpha i,$$ and the result follows from (P3) by letting $\alpha\to0$.
	
\end{proof}

\begin{lemma}\label{positive}
 If $u\in\mathcal S$ is positive then:
\begin{enumerate}
	\item[(i)] $\phi(e^u)=e^{\phi(u)}$,
	\item[(ii)] $\phi(u^n)=\phi(u)^n$ for each $n\in\mathbb N$.
\end{enumerate}
\end{lemma}
\begin{proof} (i) Let $\alpha >0$. Then $\phi(u+\alpha \mathbf 1)\neq 0$ so that Lemma~\ref{lem31} gives $\phi(e^{u+\alpha \mathbf 1})=e^{\phi(u+\alpha \mathbf 1)}$. Taking the limit as $\alpha\rightarrow0$, using (P3), we obtain $\phi(e^u)=e^{\phi(u)}$.\\
(ii) Let $\alpha>0$. Then $\phi(u+\alpha\mathbf 1)\neq0$ whence $\phi\left((u+\alpha\mathbf 1)^n\right)=\phi(u+\alpha\mathbf 1)^n$. So taking the limit as $\alpha\rightarrow0$ we get $\phi(u^n)=\phi(u)^n.$	
\end{proof}

\begin{lemma}\label{zerocase}
Let $x\in\mathcal S$ and let $0\neq t\in\Bbb R$. Then $\phi(x)=0$ if and only if $\phi(tx)=0$.	
\end{lemma}	
\begin{proof}
Suppose $\phi(x)=0$ but $\phi(tx)\not=0$. Let $\alpha>0$. Notice first that $\phi(\alpha i\mathbf 1+x)\not=0$ since 
$\alpha i\mathbf 1+x\in G(A)$. Arguing as before we have
\begin{equation}\label{z1}
\phi(tx)\phi(\alpha i\mathbf 1+x)=\alpha t\beta_\alpha i+t\beta^2_\alpha\mbox{ where }0\neq\beta_\alpha\in\sigma(x).
\end{equation}
Writing $\phi(\alpha i\mathbf 1+x)=\alpha i+\lambda_\alpha$ we also have 
\begin{equation}\label{z2}
\phi(tx)\phi(\alpha i\mathbf 1+x)=\phi(tx)\alpha i+\phi(tx)\lambda_\alpha\mbox{ where }\lambda_\alpha\in\sigma(x).
\end{equation}
Comparing the imaginary parts of \eqref{z1} and \eqref{z2} we see that $t\beta_\alpha=\phi(tx)$ for each $\alpha$. 
If we let $\alpha\to0$, then, since $\phi(x)=0$, we observe from \eqref{z1} that $\beta^2_\alpha\to0$ whence $\beta_\alpha\to0$. But this means that $\phi(tx)=0$, contradicting the assumption. For the reverse implication, if $\phi(tx)=0$, then the preceding argument implies that
$\phi(x)=\phi(tx/t)=0.$
\end{proof}

\begin{lemma}\label{square}
Let $x\in \mathcal{S}$. If $\phi(x)=0$, then $\phi\left(x^2\right)=0$.
\end{lemma}
\begin{proof}
	Suppose to the contrary that $\phi(x^2)\neq 0$. Let $\alpha>0$. As before
	\begin{equation}\label{c1}
	 \phi\left(x^2\right)\phi\left(x+\alpha i\mathbf1\right)=\beta^3_\alpha+\alpha\beta^2_\alpha i,
	 \end{equation}
	  where $\beta_\alpha\in\sigma(x)$ for each $\alpha$. If we write $\phi\left(x+\alpha i\mathbf 1\right)=\lambda_\alpha+\alpha i$ then it follows that
	 \begin{equation}\label{c2}
	  \phi\left(x^2\right)\phi\left(x+\alpha i\mathbf1\right)=\lambda_\alpha\phi\left(x^2\right)+\alpha\phi\left(x^2\right)i,
	  \end{equation}
	   where $\lambda_\alpha\in\sigma(x)$ for each $\alpha$. Comparing \eqref{c1} and \eqref{c2}, using $\phi\left(x^2\right)\not=0$, we obtain $\phi\left(x^2\right)=\beta^2_\alpha$  and $\lambda_\alpha=\beta_\alpha$ for each $\alpha$. With (P3) we have that
	   $$\lim_{\alpha\rightarrow0}\lambda_\alpha=\lim_{\alpha\rightarrow0}\phi\left(x+\alpha i\mathbf 1\right)-\alpha i=\phi(x)=0.$$
	   	From the above relations it is then clear that $\phi\left(x^2\right)=0$ which contradicts the  assumption. 
\end{proof} 

\begin{lemma}\label{selfexp}
 Let $x\in \mathcal{S}$, and let $t\in\mathbb R$. Then $\phi\left(e^{tx}\right)=e^{\phi(tx)}=e^{t\phi(x)}$.
\end{lemma}
\begin{proof}
	If $t=0$ the result is clear; so we assume $t\not=0$. From Lemma~\ref{lem31}(iii) and Lemma~\ref{zerocase}, it suffices to prove the result for the case $\phi(x)=0=\phi(tx)$. In particular it remains to prove that $\phi\left(e^{tx}\right)=1$. Consider $ \phi(e^{tx})\phi(\mathbf 1+i x)\in \sigma(e^{tx}+ie^{tx} x)$ with $\phi(\mathbf 1+i x)=1+i\lambda$ where $\lambda \in \sigma(x)$. We then have, using the Spectral Mapping Theorem as before, that $$\phi(e^{tx})+i\phi(e^{tx})\lambda=e^{t\beta}+i\beta e^{t\beta},$$
	for some $\beta\in\sigma(x).$ By comparison we see that $\lambda=\beta$ and thus $\phi(e^{tx})=e^{t\lambda}$. A similar argument with $\phi\left(e^{(tx)^2}\right)\phi(\mathbf 1+ix)$ yields $\phi\left(e^{(tx)^2}\right)=e^{(t\lambda)^2}$. Since $(tx)^2$ is positive Lemma~\ref{positive} gives $\phi\left(e^{(tx)^2}\right)=e^{\phi\left((tx)^2\right)}$. But, by Lemma~\ref{square}, $\phi(tx)=0\Rightarrow\phi\left((tx)^2\right)=0$ from which we conclude that $e^{(t\lambda)^2}=1$. This implies that $\lambda^2=0$ and hence $\phi(e^{tx})=e^{t\lambda}=1$.
	\end{proof}

\begin{lemma}\label{finduc}
 $\phi$ has the following properties:
	\begin{itemize}
\item[(i)]{ If $x\in\mathcal S$, then $\phi\left(e^{\lambda x}\right)=e^{\lambda\phi(x)}$ holds for all $\lambda\in\mathbb C$.}		
\item[(ii)]{If $x,x_1,\dots, x_n\in\mathcal S$, and $\lambda\in\mathbb C$, then 
	$$\phi\left(e^{\lambda x}e^{x_1}\cdots e^{x_{n}}\right)=\phi\left(e^{\lambda x}\right)\phi\left(e^{x_{1}}\right)\cdots \phi\left(e^{x_{n}}\right).$$   }	
\item[(iii)]{If $x,y\in \mathcal{S}$, then $\phi(x+y)=\phi(x)+\phi(y)$.}
	\end{itemize}
\end{lemma}
\begin{proof}
 We first derive equation \eqref{transinv} which will be used throughout the remainder of the proof:
   Let $x\in\mathcal S$ and let $\alpha\in\mathbb R$. By the assumption on $\phi$ we have that
$$\phi\left(e^{x+\alpha\mathbf 1}\right)\phi\left(e^{-x}\right)\in\sigma\left(e^{\alpha\mathbf 1}\right)=\{e^\alpha\}.$$ 
Hence, using Lemma~\ref{selfexp}, it follows that 
$$e^{\phi(x+\alpha\mathbf 1)-\phi(x)}=e^\alpha$$ which implies that 
\begin{equation}\label{transinv}
\phi(x+\alpha\mathbf 1)=\phi(x)+\alpha\ \mbox{ if }\alpha\in\mathbb R.
\end{equation}

(i): If $\lambda\in\mathbb C$, then, by the hypothesis on $\phi$, 
we have that 
\begin{equation}\label{alt}
\phi\left(e^{\lambda x}\right)\phi\left(e^{-\Rea(\lambda)x}\right)\in\sigma\left(e^{\Ima(\lambda)xi}\right).
\end{equation}	
We may write $\phi(e^{\lambda x})=e^{\lambda\alpha_\lambda}$ where $\alpha_\lambda\in\sigma(x)$ depends on $\lambda$. So, using Lemma~\ref{selfexp},
$$\phi\left(e^{\lambda x}\right)\phi\left(e^{-\Rea(\lambda)x}\right)=e^{\Rea(\lambda)[\alpha_\lambda-\phi(x)]}e^{\Ima(\lambda)\alpha_\lambda i},$$
and on the other hand, using \eqref{alt},
$$\phi\left(e^{\lambda x}\right)\phi\left(e^{-\Rea(\lambda)x}\right)=e^{\Ima(\lambda)\beta_\lambda i}$$ where 
 $\beta_\lambda\in\sigma(x)$ depends on $\lambda$. This forces
 $\Rea(\lambda)[\alpha_\lambda-\phi(x)]=0$ from which  
 $\alpha_\lambda=\phi(x)$ when $\Rea(\lambda)\not=0$. We may therefore conclude that
 $\phi\left(e^{\lambda x}\right)=e^{\lambda\phi(x)}$ holds whenever $\Rea(\lambda)\not=0$ and hence, by (P3),  $\phi\left(e^{\lambda x}\right)=e^{\lambda\phi(x)}$ holds for all $\lambda\in\mathbb C$. \\
(ii): Take $x,y\in\mathcal S$ and assume that $\phi(x)=\phi(y)=0$. Let $\lambda\in\mathbb C$ arbitrary. By the hypothesis on $\phi$ we have that
$$\phi\left(e^{-\lambda x}\right)\phi\left(e^{\lambda x}e^y\right)\in\sigma\left(e^y\right).$$ From part (i) of the proof, together with the assumption that $\phi(x)=0$, it then follows that $\phi\left(e^{-\lambda x}\right)=1$ whence 
$\phi\left(e^{\lambda x}e^y\right)\in\sigma\left(e^y\right).$ So we can write $\phi\left(e^{\lambda x}e^y\right)=e^{\alpha_\lambda}$ where 
$\alpha_\lambda\in\sigma(y)$ depends on $\lambda$. Similarly, since $\phi(e^{-y})=1$, it follows from the hypothesis on $\phi$ that
$$\phi\left(e^{\lambda x}e^y\right)=\phi\left(e^{\lambda x}e^y\right)\phi(e^{-y})\in\sigma\left(e^{\lambda x}\right),$$
and so we can write $\phi\left(e^{\lambda x}e^y\right)=e^{\lambda\beta_\lambda}$ where $\beta_\lambda\in\sigma(x)$ depends on $\lambda$. 
Together we have that $e^{\alpha_\lambda}=e^{\lambda\beta_\lambda}$ for all $\lambda\in\mathbb C$. If we then write
\begin{equation}\label{alt2}
e^{\alpha_\lambda}=e^{\Rea(\lambda)\beta_\lambda+\Ima(\lambda)\beta_\lambda\,i}=e^{\Rea(\lambda)\beta_\lambda}e^{\Ima(\lambda)\beta_\lambda\,i}
\end{equation}
 it follows, from the fact that $\alpha_\lambda,\beta_\lambda\in\mathbb R$, that $\alpha_\lambda=\Rea(\lambda)\beta_\lambda$. Since $\lambda\mapsto \alpha_\lambda$ is continuous this proves that the function $\lambda\mapsto\beta_\lambda$ is continuous when $\Rea(\lambda)\not=0$. Further, since \eqref{alt2} forces $e^{\Ima(\lambda)\beta_\lambda\,i}\in\mathbb R$, we have that
$$\sin(\Ima(\lambda)\beta_\lambda)=0\Rightarrow\Ima(\lambda)\beta_\lambda=k_\lambda\pi$$ where $\lambda\mapsto k_\lambda$ is a continuous function on $\Rea(\lambda)\not=0$. Since $k_\lambda$ takes values in $\mathbb Z$ it must be constant, and in particular, since $\beta_\lambda$ is bounded, $k_\lambda=0$ when $\Rea(\lambda)\not=0$. So we observe that $\beta_\lambda=0$ when both $\Rea(\lambda)\not=0$, and $\Ima(\lambda)\not=0$. Thus $e^{\alpha_\lambda}=1$ when $\Rea(\lambda)\not=0$, and $\Ima(\lambda)\not=0$, and by (P3) we find that $\phi\left(e^{\lambda x}e^y\right)=1$ holds for all $\lambda\in\mathbb C$.
Now if we replace $x,y\in\mathcal S$ by respectively $x-\phi(x)\mathbf 1$ and $y-\phi(y)\mathbf 1$ then both elements belong to $\mathcal S$ and, by \eqref{transinv}, both are in the kernel of $\phi$. By the preceding paragraph we then have, using \eqref{obs2}, 
\begin{align*}
1&=\phi\left(e^{\lambda(x-\phi(x)\mathbf 1)}e^{y-\phi(y)\mathbf 1}\right)=\phi\left(e^{-[\lambda\phi(x)+\phi(y)]\mathbf 1}e^{\lambda x}e^y\right)\\&=
\phi\left(e^{-[\lambda\phi(x)+\phi(y)]}e^{\lambda x}e^y\right)=e^{-[\lambda\phi(x)+\phi(y)]}\phi\left(e^{\lambda x}e^y\right)
\end{align*}  
from which we obtain $\phi\left(e^{\lambda x}e^y\right)=\phi\left(e^{\lambda x})\phi(e^y\right)$.
To extend the preceding formula, let us assume, for any collection of $n+1$ elements $\{x,x_1,\dots,x_n\}\subset S$, it holds that
$$\phi\left(e^{\lambda x}e^{x_1}\cdots e^{x_n}\right)=\phi\left(e^{\lambda x}\right)\phi\left(e^{x_1}\right)\cdots\phi\left(e^{x_n}\right)\ \ (\lambda\in\mathbb C).$$
 Take  $\{x,x_1,\dots,x_n, x_{n+1}\}\subset S$ and assume 
\begin{equation}\label{assume}
\phi(x)=\phi(x_1)=\cdots=\phi(x_{n+1})=0.
\end{equation}
Observe, from the hypothesis on $\phi$, that
$$\phi\left(e^{\lambda x}e^{x_1}\cdots e^{x_{n+1}}\right)\phi\left(e^{-x_{n+1}}\cdots e^{-x_1}\right)\in\sigma\left(e^{\lambda x}\right),$$
and then, from the induction assumption together with \eqref{assume} and (i), that 
$$\phi\left(e^{\lambda x}e^{x_1}\cdots e^{x_{n+1}}\right)\in\sigma\left(e^{\lambda x}\right).$$
So we can write 
$$\phi\left(e^{\lambda x}e^{x_1}\cdots e^{x_{n+1}}\right)=e^{\lambda\beta_\lambda}\mbox{ where }\beta_\lambda\in\sigma(x).$$
On the other hand we also have 
$$\phi\left(e^{-x_n}e^{-x_{n-1}}\cdots e^{-x_1}e^{-\lambda x}\right)\phi\left(e^{\lambda x}e^{x_1}\cdots e^{x_{n+1}}\right)\in\sigma\left(e^{x_{n+1}}\right),$$
and, again using the induction assumption together with \eqref{assume} and (i),
$$\phi\left(e^{\lambda x}e^{x_1}\cdots e^{x_{n+1}}\right)=e^{\alpha_\lambda}\mbox{ where }\alpha_\lambda\in\sigma(x_{n+1}).$$
Using the same argument that was used for $e^{\lambda x}e^y$ we obtain
$$\phi\left(e^{\lambda x}e^{x_1}\cdots e^{x_{n+1}}\right)=1.$$
If we consequently replace the collection $\{x,x_1,\dots,x_n, x_{n+1}\}$ by $$\{x-\phi(x)\mathbf 1,x_1-\phi(x_1)\mathbf 1,\dots,x_n-\phi(x_n)\mathbf 1, x_{n+1}-\phi(x_{n+1})\mathbf 1\},$$
and use the argument that was used for $e^{\lambda x}e^y$ we arrive at 
$$\phi\left(e^{\lambda x}e^{x_1}\cdots e^{x_{n+1}}\right)=\phi\left(e^{\lambda x}\right)\cdots \phi\left(e^{x_{n+1}}\right).$$
So the result follows by induction.\\
(iii): If $n\in\mathbb N$ then, from (ii), it follows that 
$$\phi\left(\left[e^{x/n}e^{y/n}\right]^n\right)=e^{\phi(x)+\phi(y)}.$$
But, by (P3), we have 
$$\lim_n\phi\left(\left[e^{x/n}e^{y/n}\right]^n\right)=\phi\left(\lim_n\left[e^{x/n}e^{y/n}\right]^n\right)=\phi\left(e^{x+y}\right)=e^{\phi(x+y)}.$$
Since $\phi$ takes real values on $\mathcal S$ we have the result. 
\end{proof}

\begin{theorem}\label{Cstar}
  The formula $$\psi_\phi(x):=\phi\left(\Rea(x)\right)+i\phi\left(\Ima(x)\right)$$ defines a character on $A$.
\end{theorem}
\begin{proof}
	By Lemma~\ref{finduc}(iii), together with the Kowalski-S\l{}odkowski Theorem, $\psi_\phi$ would be a character if we can prove that $\psi_\phi(x)\in\sigma(x)$ for each $x\in A$. Write $x=u+i\,v$ where $u:=\Rea(x)$ and $v:=\Ima(x)$. By the hypothesis on $\phi$ it follows that
	$$\left[\phi\left(e^{tu}\right)\phi\left(e^{itv}\right)-1\right]/t\in\sigma\left(\left[e^{tu}e^{itv}-\mathbf 1\right]/t\right).$$
	Hence, by Lemma~\ref{finduc}(i), we have that 
	$$\left[e^{t\phi(u)}e^{it\phi(v)}-1\right]/t\in\sigma\left(\left[e^{tu}e^{itv}-\mathbf 1\right]/t\right).$$
	If we let $t\rightarrow 0$, then, using the fact that $A\setminus G(A)$ is closed in $A$, it follows that
	 $\phi(u)+i\phi(v)\in\sigma(u+i\,v)$.
\end{proof}	

\begin{lemma}\label{lastlem}
If $x_1,\dots, x_n\in\mathcal S$, and $\lambda_1,\dots,\lambda_n\in\mathbb C$, then 
\begin{align*}
\phi\left(e^{\lambda_1 x_1}\cdots e^{\lambda_nx_{n}}\right)&=\phi\left(e^{\lambda_1 x_1}\right)\cdots \phi\left(e^{\lambda_nx_{n}}\right)\\&=e^{\lambda_1\phi(x_1)}\cdots e^{\lambda_n\phi(x_n)}.
\end{align*}
\end{lemma}
\begin{proof}
Let $x,y\in\mathcal S$ such that $\phi(x)=\phi(y)=0$. For any $\lambda\in\mathbb C$ consider the expression 
$\phi\left(e^{\gamma x}e^{\lambda y}\right)$ where $\gamma\in\mathbb C$ is arbitrary but fixed. By the assumption on $\phi$ we have 
$$\phi\left(e^{\gamma x}e^{\lambda y}\right)\phi\left(e^{-\gamma x}\right)\in\sigma\left(e^{\lambda y}\right),$$ and so,
by Lemma~\ref{finduc}(i), $\phi\left(e^{\gamma x}e^{\lambda y}\right)=e^{\lambda\alpha_{\lambda}}$ where $\alpha_\lambda$ is a function of $\lambda$ with values belonging to $\sigma(y)$. Again by the assumption on $\phi$ we have that 
$$\phi\left(e^{\gamma x}e^{\lambda y}\right)\phi\left(e^{-\Rea(\lambda)y}e^{-\gamma x}\right)\in\sigma\left(e^{\Ima(\lambda)yi}\right),$$ and so,  
since $-\Rea(\lambda)y\in\mathcal S$, Lemma~\ref{finduc}(i) and (ii) give
$\phi\left(e^{\gamma x}e^{\lambda y}\right)=e^{\Ima(\lambda)\beta_{\lambda}i}$ where $\beta_\lambda$ is a function of $\lambda$ with values belonging to $\sigma(y)$. Hence we obtain
$$e^{\Ima(\lambda)\beta_{\lambda}i}=e^{\Rea(\lambda)\alpha_\lambda}e^{\Ima(\lambda)\alpha_\lambda i},$$ from which it follows that $\alpha_\lambda=0$ whenever $\Rea(\lambda)\not=0$. Consequently $\phi\left(e^{\gamma x}e^{\lambda y}\right)=1$ for all $\lambda$ with $\Rea(\lambda)\not=0$ which extends to $\phi\left(e^{\gamma x}e^{\lambda y}\right)=1$ for all $\lambda\in\mathbb C$ via (P3). Now if we replace $x,y\in\mathcal S$ by respectively $x-\phi(x)\mathbf 1$ and $y-\phi(y)\mathbf 1$ then both elements belong to $\mathcal S$ and, by \eqref{transinv}, both are in the kernel of $\phi$. Exactly as in the proof of Lemma~\ref{finduc} we obtain  $$\phi\left(e^{\gamma x}e^{\lambda y}\right)=\phi\left(e^{\gamma x}\right)\phi\left(e^{\lambda y}\right)=e^{\gamma\phi(x)}e^{\lambda\phi(y)}.$$
 Assume now, for any collection of $n$ elements, $\lambda_1,\dots,\lambda_n\in\mathbb C$, and $n$ elements, $x_1,\dots,x_n\in\mathcal S$ it holds that 
 $$\phi\left(e^{\lambda_1x_1}\cdots e^{\lambda_nx_n}\right)=
 \prod_{i=1}^n\phi\left(e^{\lambda_ix_i}\right)=\prod_{i=1}^ne^{\lambda_i\phi(x_i)}.$$
 Let $z\in\mathcal S$. Consider the expression 
 $\phi\left(e^{\lambda x_1}\cdots e^{\lambda_nx_n}e^z\right)$ where $\lambda$ is a variable and $\lambda_2,\dots,\lambda_n$ are fixed. Suppose that $\phi(x_i)=0=\phi(z)$. 
 Then, from the induction hypothesis,
   $$\phi\left(e^{\lambda x_1}\cdots e^{\lambda_nx_n}e^z\right)=\phi\left(e^{-\lambda_n x_n}\cdots e^{-\lambda x_1}\right)\phi\left(e^{\lambda x_1}\cdots e^{\lambda_nx_n}e^z\right)=e^{\alpha_\lambda}$$ where 
 $\alpha_\lambda\in\sigma(z)$. 
 On the other hand, again using the induction hypothesis, we also have that 
  $$\phi\left(e^{\lambda x_1}\cdots e^{\lambda_nx_n}e^z\right)=\phi\left(e^{\lambda x_1}\cdots e^{\lambda_nx_n}e^z\right)\phi\left(e^{-z}e^{-\lambda_nx_n}\cdots e^{-\lambda_2x_2}\right)=e^{\lambda\omega_\lambda}$$ where 
  $\omega_\lambda\in\sigma(x_1)$. It thus follows, by using the same argument following equation \eqref{alt2} in the proof of Lemma~\ref{finduc}, that $\phi\left(e^{\lambda x_1}\cdots e^{\lambda_nx_n}e^z\right)=1$ for each $\lambda\in\mathbb C$.  Replacing  $\{x_1,\dots,x_n, z\}$ with $$\{x_1-\phi(x_1)\mathbf 1,\dots,x_n-\phi(x_n)\mathbf 1, z-\phi(z)\mathbf 1\},$$ as in the proof of Lemma~\ref{finduc}, we have, setting $\lambda=\lambda_1,$ that 
    $$\phi\left(e^{\lambda_1x_1}\cdots e^{\lambda_nx_n}e^z\right)=\left(\prod_{i=1}^ne^{\lambda_i\phi(x_i)}\right)e^{\phi(z)}.$$
  With the same induction hypothesis, let $x_1,\dots,x_n,x_{n+1}\in\mathcal S$, let $\lambda_1,\dots,\lambda_{n},\lambda\in\mathbb C$, where $\lambda$ is variable, and suppose $\phi(x_i)=0$ for $i=1,\dots,n+1.$ 
  Arguing as before, we have  $$\phi\left(e^{\lambda_1x_1}\cdots e^{\lambda_nx_n}e^{\lambda x_{n+1}}\right)\in\sigma\left(e^{\lambda x_{n+1}}\right)$$
  from which it follows that 
  $$\phi\left(e^{\lambda_1x_1}\cdots e^{\lambda_nx_n}e^{\lambda x_{n+1}}\right)=e^{\lambda\gamma_\lambda}$$
  where $\gamma_\lambda\in\sigma(x_{n+1})$. On the other hand, using the result derived in the preceding paragraph, we also have 
  $$\phi\left(e^{\lambda_1x_1}\cdots e^{\lambda_nx_n}e^{\lambda x_{n+1}}\right)=e^{\Ima(\lambda)\eta_\lambda i}$$
  where $\eta_\lambda\in\sigma(x_{n+1})$, and, using the same argument following equation \eqref{assume} in the proof of Lemma~\ref{finduc}, we obtain 
    $$\phi\left(e^{\lambda_1x_1}\cdots e^{\lambda_nx_n}e^{\lambda x_{n+1}}\right)=1.$$
    Replacing $x_i$ by $x_i-\phi(x_i)\mathbf 1$ for each $i$, using \eqref{obs2}, and setting $\lambda=\lambda_{n+1}$ we then deduce
   $$\phi\left(e^{\lambda_1x_1}\cdots e^{\lambda_{n+1}x_{n+1}}\right)=
   \prod_{i=1}^{n+1}\phi\left(e^{\lambda_ix_i}\right)=\prod_{i=1}^{n+1}e^{\lambda_i\phi(x_i)}.$$ So the result follows by induction.
\end{proof}

\begin{theorem}
	$\psi_\phi$ agrees with $\phi$ on $G_{\mathbf 1}(A)$. 
\end{theorem}

\begin{proof}
For $j=1,\dots,m$ let $x_j=u_j+iv_j$ where $u_j:=\Rea(x_j)$ and $v_j:=\Ima(x_j)$. By Lemma~\ref{lastlem} we have
\begin{equation}\label{product}
\phi\left[\prod_{j=1}^m\left(e^{u_j/n}e^{iv_j/n}\right)^n\right]
=\prod_{j=1}^me^{\phi(u_j)+i\phi(v_j)}.
\end{equation} 
Taking the limit as $n\rightarrow\infty$ on the left side of \eqref{product} gives
\begin{align*}
\phi\left(\prod_{j=1}^me^{x_j}\right)&=\prod_{j=1}^me^{\phi(u_j)+i\phi(v_j)}=\prod_{j=1}^me^{\psi_\phi(x_j)}=\psi_\phi\left(\prod_{j=1}^me^{x_j}\right),
\end{align*}
and the result is clear from \eqref{G1}.	
\end{proof}

 \bibliographystyle{amsplain}
 \bibliography{Spectral}

\end{document}